\documentclass{amsart}

\newtheorem{thm}{Theorem}[section]
\newtheorem{prp}[thm]{Proposition}
\newtheorem{lem}[thm]{Lemma}
\newtheorem{cor}[thm]{Corollary}
\newtheorem{exa}[thm]{Example}

\numberwithin{equation}{section}

\newcommand \z      {\mathbf z}
\newcommand \N      {\mathbb N}
\newcommand \Z      {\mathbb Z}
\newcommand \R      {\mathbb R}
\newcommand \F      {\mathcal F}
\newcommand \pt     {\mathbf{pt}}
\newcommand \id     {\mathrm{id}}
\newcommand \susp   {\Sigma}
\newcommand \Tor    {\operatorname{Tor}}
\newcommand \betti  {\widetilde \beta}
\newcommand \sub[1] {\langle #1 \rangle}
\newcommand \cpx[1] {\lVert #1 \rVert}

\begin{document}

\title{Homotopy type of Frobenius complexes II}

\author[TOUNAI S.]{TOUNAI Shouta}

\address{Graduate School of Mathematical Sciences, The University of Tokyo, 3-8-1 Komaba, Meguro, Tokyo, 153-8914 Japan}

\email{tounai@ms.u-tokyo.ac.jp}



\begin{abstract}
  A finitely generated additive submonoid $\Lambda$ of $\N^d$ has the partial order defined by
  $\lambda \le \lambda + \mu$ for $\lambda, \mu \in \Lambda$.
  The Frobenius complex is the order complex of an open interval of $\Lambda$.
  In this paper, we express the homotopy type of the Frobenius complex of $\Lambda[\rho / r]$
  in terms of those of $\Lambda$,
  where $\Lambda[\rho / r]$ is the additive monoid which is added the $r$-th part of $\rho$ to $\Lambda$.
  As applications, we determine the homotopy type of the Frobenius complex of some submonoids of $\N$,
  for example, the submonoid generated by a finite geometric sequence.
  We also state an application to the multi-graded Poincar\'e series.
\end{abstract}

\maketitle

\section{Introduction}

We consider a finitely generated additive monoid $\Lambda$
which is cancellative and has no non-zero invertible element,
for example, a finitely generated additive submonoid of $\N^d$.
Then the partial order $\le_\Lambda$ on $\Lambda$ is defined by
$\lambda \le_\Lambda \lambda + \mu$ for $\lambda, \mu \in \Lambda$.
The Frobenius complex $\F(\lambda; \Lambda)$ is the order complex $\cpx{(0, \lambda)_\Lambda}$
of the open interval of $\Lambda$.

Laudal and Sletsj{\o}e~\cite{LS} proved that the graded component $\Tor_{i, \lambda}^{K[\Lambda]}(K, K)$
of the torsion group over the monoid algebra $K[\Lambda]$ is isomorphic to
the reduced homology group $\widetilde H_{i - 2}(\F(\lambda; \Lambda); K)$ of the Frobenius complex
with coefficients in $K$,
and showed, as an application, that the Poincar\'e  series of a saturated rational submonoid $\Lambda$ of $\N^2$
is given by a rational function.

The multi-graded Poincar\'e series $P_\Lambda(t, \z)$ of $\Lambda$ is defined by
\[ P_\Lambda(t, \z) = \sum_{i \in \N} \sum_{\lambda \in \Lambda} \dim_K \Tor_{i, \lambda}^{K[\Lambda]}(K, K) \cdot t^i \z^\lambda, \]
and it is an open question whether $P_\Lambda(t, \z)$ is given by a rational function or not
for an additive monoid $\Lambda$ (see \cite{PRS}).

Clark and Ehrenborg~\cite{CE} determined the homotopy type of the Frobenius complexes of some additive submonoids of $\N$.
For example, it is shown that the Frobenius complex of the submonoid of $\N$ generated by two elements
is either contractible or homotopy equivalent to a sphere,
and the multi-graded Poincar\'e series is determined and proved to be rational.
The proof is based on discrete Morse theory.

In this paper, we establish a method to calculate the homotopy type of the Frobenius complexes.
We construct the additive monoid $\Lambda[\rho / r]$ which is added the $r$-th part of $\rho$ to $\Lambda$,
and show that for $\lambda \in \Lambda$ and $k \in \N^{< r}$ the Frobenius complex of $\Lambda[\rho / r]$ satisfies
\[ \F(\lambda + k \rho / r; \Lambda[\rho / r]) \simeq
\begin{cases}
  \bigvee\limits_{\ell \rho \le_\Lambda \lambda} \susp^{2 \ell + k} \F(\lambda - \ell \rho; \Lambda)
  & \text{if $k \le 1$,} \\
  \pt & \text{if $k \ge 2$,}
\end{cases} \]
when $\rho$ is a reducible element of $\Lambda$ and $r \ge 2$.
This is the main theorem of this paper.
The proof is based on the basic technique of homotopy theory for posets and CW complexes.
Applying the main theorem, we derive a formula for the multi-graded Poincar\'e series of $\Lambda[\rho / r]$, namely
\[ P_{\Lambda[\rho / r]}(t, \z) = \frac{P_\Lambda(t, \z) \cdot (1 + t \z^{\rho / r})}{1 - t^2 \z^\rho}. \]
As applications, we show that the Frobenius complex $\F(\lambda; \Lambda)$ of $\Lambda$
is homotopy equivalent to a wedge of spheres (the dimensions may not coincide),
and determine the multi-graded Poincar\'e series of $\Lambda$, which is proved to be rational, for some submonoids $\Lambda$ of $\N$,
for example, the submonoid of $\N$ generated by the geometric sequence
$p^n, p^{n - 1} q, \dots, p q^{n - 1}, q^n$.
This gives an answer to a question raised by Clark and Ehrenborg (\cite{CE}, Question 6.4).

This paper is organized in the following way.
Section~\ref{sec : pre} presents some preliminaries.
In Section~\ref{sec : main} the main theorem is stated and proved.
Section~\ref{sec : app} is devoted to applications to the multi-graded Poincar\'e series.
In Section~\ref{sec : exa} we calculate some examples using the results of previous sections.

\section{Preliminaries} \label{sec : pre}

\subsection{Topology} \label{subsec : topology}

In this paper, $S^n$ denotes the $n$-sphere for $n \ge 0$, namely
\[ S^n = \{\, (x_0, \dots, x_n) \in \R^{n + 1} \mid {x_0}^2 + \dots + {x_n}^2 = 1 \,\}, \]
$S^{-1}$ the empty space, and $\pt$ the one-point space.
We introduce a formal symbol $S^{-2}$.
Fix a field $K$.
We define the \emph{reduced Betti number} $\betti_i(X)$ for a non-empty topological space $X$ and $i \in \Z$ by
\[ \betti_i(X) =
  \begin{cases}
    \dim_K \widetilde H_i(X; K) & \text{if $i \ge 0$,} \\
    0 & \text{if $i < 0$.} \\
  \end{cases} \]
For the empty space $S^{-1}$, we define
\[ \betti_i(S^{-1}) = \delta_{i, -1} \qquad (i \in \Z), \]
where $\delta_{i, j}$ denotes the Kronecker's delta.
We also define
\[ \betti_i(S^{-2}) = \delta_{i, -2} \qquad (i \in \Z). \]

By $\susp X$ we denote the (unreduced) suspension of a topological space $X$.
Note that $\susp S^{-1} \approx S^0$.
We also define  $\susp S^{-2} = S^{-1}$.
Then
\[ \betti_i(\susp X) = \betti_{i - 1}(X) \qquad (i \in \Z) \]
holds when $X$ is either a topological space or $S^{-2}$.

We use the two following lemmas in the proof of the main theorem.

\begin{lem} \label{lem : susp}
  Let $X_1$ and $X_2$ be subcomplexes of a CW complex $X$.
  If both $X_1$ and $X_2$ are contractible,
  then the union $X_1 \cup X_2$ is homotopy equivalent to the suspension $\susp (X_1 \cap X_2)$ of the intersection.
\end{lem}

\begin{lem} \label{lem : wedge}
  Let $X_1$ and $X_2$ be subcomplexes of a CW complex $X$.
  If $X_2$ is contractible and if the inclusion $X_1 \cap X_2 \hookrightarrow X_1$ is
  homotopic to the constant map to a point $x_1$ of $X_1$,
  then the union $X_1 \cup X_2$ is homotopy equivalent to the wedge $X_1 \vee \susp (X_1 \cap X_2)$
  of $X_1$ and the suspension of the intersection,
  where we let $X_1$ be pointed at $x_1$ and $\susp (X_1 \cap X_2)$ at one end point.
\end{lem}

\subsection{Posets}

A \emph{partially ordered set} (\emph{poset}, in short) is a set $P$ together with a partial order $\le_P$ on $P$.
By $<_P$ we denote the strict order associated to $\le_P$, namely
\[ x <_P y \iff \text{$x \le_P y$ and $x \neq y$} \qquad (x, y \in P). \]
By $\cpx P$ we denotes the order complex of $P$,
that is, $\cpx P$ is the abstract simplicial complex whose vertices are elements of $P$
and whose simplices are finite chains of $P$.
We also denote the geometric realization of the order complex $\cpx P$ by the same symbol $\cpx P$.

Let $P$ and $Q$ be posets.
A \emph{poset map} from $P$ to $Q$ is a set map $f : P \to Q$ such that
$x \le_P y$ implies $f(x) \le_Q f(y)$ for any $x, y \in P$.
By $\cpx f$ we denotes the continuous map from $\cpx P$ to $\cpx Q$ induced by $f$.
For poset maps $f, g : P \to Q$ we write $f \le g$ if $f(x) \le_Q g(x)$ holds for each $x \in P$.
Note that if $f \le g$, then the induced map $\cpx f$ is homotopic to $\cpx g$.

Let $P$ be a poset and $a, b \in P$.
We write
\begin{align*}
  P_{\ge a} &= \{\, x \in P \mid a \le_P x \,\} \\
  P_{> a} &= \{\, x \in P \mid a <_P x \,\} \\
  [a, b]_P &= \{\, x \in P \mid a \le_P x \le_P b \,\}, \quad \text{and} \\
  (a, b)_P &= \{\, x \in P \mid a <_P x <_P b \,\}.
\end{align*}
Similarly, $P^{\le b}, P^{< b}, [a, b)_P$ and $(a, b]_P$ are defined.

The following lemmas are useful to calculate homotopy types of posets.

\begin{lem}
  If a poset $P$ has either a minimum element or a maximum element,
  then the order complex $\cpx P$ is contractible.
\end{lem}

\begin{proof}
  We show only the case where $P$ has a minimum element $m$.
  Let $c : P \to P$ be the constant map to $m$.
  Then $c$ is a poset map and satisfies $c \le \id_P$.
  Thus the identity map of $\cpx P$ is homotopic to the constant map $\cpx c$.
  Hence $\cpx P$ is contractible.
\end{proof}

\begin{lem} \label{lem : poset}
  Let $P$ and $Q$ be posets, $f : P \to Q$ and $g : Q \to P$ poset maps, and $a \in P$.
  Assume that $g f \le \id_P$ and $f g \le \id_Q$.
  Then we have
  \[ \cpx{P^{< a}} \simeq
    \begin{cases}
      \cpx{Q^{< f(a)}} & \text{if $g f(a) = a$,} \\
      \pt & \text{if $g f(a) <_P a$.}
    \end{cases} \]
\end{lem}

\begin{proof}
  We first show the case $g f(a) <_P a$.
  We have
  \begin{align*}
    &f(P^{< a}) \subset f(P^{\le a}) \subset Q^{\le f(a)} \quad \text{and} \\
    &g(Q^{\le f(a)}) \subset P^{\le g f(a)} \subset P^{< a}.
  \end{align*}
  Thus $f$ and $g$ induce poset maps between $P^{< a}$ and $Q^{\le f(a)}$.
  Moreover, the induced maps between the order complexes are homotopy inverses of each other.
  Hence
  \[ \cpx{P^{< a}} \simeq \cpx{Q^{\le f(a)}} \simeq \pt. \]

  We next show the case $g f(a) = a$.
  We have
  \begin{align*}
    &f(P^{\le a}) \subset Q^{\le f(a)} \quad \text{and} \\
    &g(Q^{\le f(a)}) \subset P^{\le g f(a)} = P^{\le a}.
  \end{align*}
  Let $x \in P^{\le a}$ and $f(x) = f(a)$.
  Then we have $a = g f(a) = g f(x) \le_P x$, and so $x = a$.
  Conversely, let $y \in Q^{\le f(a)}$ and $g(y) = a$.
  Then we have $f(a) = f g(y) \le_Q y$, and so $y = f(a)$.
  Thus $f$ and $g$ induce poset maps between $P^{< a}$ and $Q^{< f(a)}$.
  Moreover, the induced maps between the order complexes are homotopy inverses of each other.
  Hence
  \[ \cpx{P^{< a}} \simeq \cpx{Q^{< f(a)}}. \qedhere \]
\end{proof}

\subsection{Frobenius complexes}

In this paper, $\N$ denotes the additive monoid of non-negative integers.
An \emph{affine monoid} is an additive monoid $\Lambda$ which satisfies three following conditions:
\begin{enumerate}
  \item $\Lambda$ is cancellative, that is, $\lambda + \nu = \mu + \nu$ implies $\lambda = \mu$
    for any $\lambda, \mu, \nu \in \Lambda$.
  \item $\Lambda$ has no non-zero invertible element, that is,
    $\lambda + \mu = 0$ implies $\lambda = \mu = 0$ for any $\lambda, \mu \in \Lambda$.
  \item $\Lambda$ is finitely generated, that is, there exist finite elements $\alpha_1, \dots, \alpha_d$ of $\Lambda$
    such that for any element $\lambda$ of $\Lambda$ can be written as
    $\lambda = m_1 \alpha_1 + \dots + m_d \alpha_d$ for some $m_1, \dots, m_d \in \N$.
\end{enumerate}
For example, a finitely generated submonoid of $\N^d$ is an affine monoid.
By $\Lambda_+$ we denote the set of all non-zero elements of $\Lambda$.

An affine monoid $\Lambda$ has the partial order $\le_\Lambda$ defined by
\[ \lambda \le_\Lambda \mu \iff \text{there exists $\nu \in \Lambda$ such that $\lambda + \nu = \mu$}
  \qquad (\lambda, \mu \in \Lambda). \]
Such an element $\nu$ is unique if exists, since $\Lambda$ is cancellative.
Thus we let $\mu - \lambda = \nu$ if $\lambda + \nu = \mu$.
By $<_\Lambda$ we denote the strict order associated to $\le_\Lambda$, namely
\begin{align*}
  \lambda <_\Lambda \mu
  &\iff \text{$\lambda \le_\Lambda \mu$ and $\lambda \neq \mu$}
  \\&\iff \text{there exists $\nu \in \Lambda_+$ such that $\lambda + \nu = \mu$}
  \qquad (\lambda, \mu \in \Lambda).
\end{align*}

For an affine monoid $\Lambda$ and $\lambda \in \Lambda$,
the \emph{Frobenius  complex} $\F(\lambda; \Lambda)$ is defined by
\[ \F(\lambda; \Lambda) =
  \begin{cases}
    \cpx{(0, \lambda)_\Lambda} & \text{if $\lambda \in \Lambda_+$,} \\
    S^{-2} & \text{if $\lambda = 0$,}
  \end{cases} \]
where $S^{-2}$ is the formal symbol introduced in Subsection~\ref{subsec : topology}.

\subsection{Multi-graded Poincar\'e series}

Let $\Lambda$ be an affine monoid, and fix a field $K$.
The \emph{multi-graded Poincar\'e series $P_\Lambda(t, \z)$} of $\Lambda$ is defined by
\[ P_\Lambda(t, \z) = P^{K[\Lambda]}_K(t, \z)
= \sum_{i \in \N} \sum_{\lambda \in \Lambda} \beta_i(\lambda; \Lambda) t^i \z^\lambda, \]
where we let
\[ \beta_i(\lambda; \Lambda) = \dim_K \operatorname{Tor}_{i, \lambda}^{K[\Lambda]}(K, K)
  \qquad (i \in \N, \ \lambda \in \Lambda). \]

Laudal and Sletsj{\o}e~\cite{LS} proved the formula
\begin{equation} \label{eqn : LS}
  \beta_i(\lambda; \Lambda) = \betti_{i - 2}(\F(\lambda; \Lambda))
  \qquad (i \in \N, \ \lambda \in \Lambda).
\end{equation}
See \cite{PRS} for more details.

\subsection{Locally finiteness of affine monoids}

\begin{prp}
  An affine monoid $\Lambda$ is locally finite,
  that is, for any $\lambda \in \Lambda$ the subset $\Lambda^{\le \lambda}$ is finite.
\end{prp}

\begin{proof}
  Take a finite generating system $\{ \alpha_1, \dots, \alpha_d \}$ of $\Lambda$.
  We can assume that each $\alpha_i$ is non-zero.
  Let $\pi : \N^d \to \Lambda$ be the map defined by
  \[ \pi(n_1, \dots, n_d) = \sum_{i = 1}^d n_i \alpha_i
    \qquad \bigl( (n_1, \dots, n_d) \in \N^d \bigr). \]
  Then $\pi$ is surjective monoid homomorphism, and satisfies that $\pi(\vec n) = 0$ implies $\vec n = \vec 0$
  for any $\vec n \in \N^d$ since $\Lambda$ has no non-zero invertible element.

  Let $\lambda \in \Lambda$.
  We now show that
  \[ \Lambda^{\le \lambda} \subset \bigcup_{\vec \ell \in \pi^{-1}(\lambda)} \pi \bigl( (\N^d)^{\le \vec \ell} \bigr). \]
  Let $\mu \in \Lambda^{\le \lambda}$.
  Take $\vec m \in \pi^{-1}(\mu)$ and $\vec n \in \pi^{-1}(\lambda - \mu)$.
  Then we have
  \[ \pi(\vec m + \vec n) = \pi(\vec m) + \pi(\vec n) = \mu + (\lambda - \mu) = \lambda. \]
  Hence
  \[ \mu = \pi(\vec m) \in \pi \bigl( (\N^d)^{\le \vec m + \vec n} \bigr)
    \subset \bigcup_{\vec \ell \in \pi^{-1}(\lambda)} \pi \bigl( (\N^d)^{\le \vec \ell} \bigr). \]

  Since each $(\N^d)^{\le \vec \ell}$ is finite, it suffices to show that $\pi^{-1}(\lambda)$ is finite.
  We now show that $\pi^{-1}(\lambda)$ is an anti-chain of $\N^d$, that is,
  $\vec m \le_{\N^d} \vec n$ implies $\vec m = \vec n$ for any $\vec m, \vec n \in \pi^{-1}(\lambda)$.
  Let $\vec m, \vec n \in \pi^{-1}(\lambda)$ and assume that $\vec m \le_{\N^d} \vec n$.
  Then we have
  \[ \lambda = \pi(\vec n) = \pi(\vec m) + \pi(\vec n - \vec m) = \lambda + \pi(\vec n - \vec m). \]
  Since $\Lambda$ is cancellative, we obtain $\pi(\vec n - \vec m) = 0$, which implies $\vec m = \vec n$.
  By the lemma below, the anti-chain $\pi^{-1}(\lambda)$ of $\N^d$ is finite.
\end{proof}

\begin{lem}
  Any anti-chain of $\N^d$ is finite.
\end{lem}

\begin{proof}
  The proof is by induction on $d$.
  If $d = 1$, any anti-chain of $\N$ is either the empty set or a singleton
  since the order of $\N$ is total.

  Let $d \ge 2$ and $A$ an anti-chain of $\N^d$.
  Assume that $A$ is not empty.
  Take $\vec a = (a_1, \dots, a_d) \in A$.
  Let $A_i = \{\, (n_1, \dots, n_d) \in A \mid n_i < a_i \,\}$.
  We now show that
  \[ A = A_1 \cup \dots \cup A_d \cup \{ \vec a \}. \]
  Let $\vec n = (n_1, \dots, n_d) \in A \setminus (\bigcup_{i = 1}^d A_i)$.
  Then we have $a_i \le n_i$ for each $i$, which implies $\vec a \le_{\N^d} \vec n$.
  Hence $\vec n = \vec a$ since $A$ is an anti-chain.

  Let $A_{i, k} = \{\, (n_1, \dots, n_d) \in A \mid n_i = k \,\}$.
  Then $A_{i, k}$ is finite
  since $A_{i, k}$ is an anti-chain of $\{\, (n_1, \dots, n_d) \in \N^d \mid n_i = k \,\} \cong \N^{d - 1}$.
  Thus each $A_i = A_{i, 0} \cup \dots \cup A_{i, a_i - 1}$ is finite.
  Hence so is $A$.
\end{proof}

\begin{cor}
  For an affine monoid $\Lambda$ and $\lambda \in \Lambda_+$
  the Frobenius complex $\F(\lambda; \Lambda)$ is a finite simplicial complex.
\end{cor}

\begin{proof}
  The open interval $(0, \lambda)_\Lambda$ is finite since $(0, \lambda)_\Lambda \subset \Lambda^{\le \lambda}$.
  Thus the order complex $\cpx{(0, \lambda)_\Lambda}$ is a finite simplicial complex.
\end{proof}

\begin{prp} \label{prp : Archi}
  Let $\Lambda$ be an affine monoid and $\rho \in \Lambda_+$.
  Then for any $\lambda \in \Lambda$
  the set of all $\ell \in \N$ satisfying $\ell \rho \le_\Lambda \lambda$ is finite.
\end{prp}

\begin{proof}
  Let $A$ be the set of all $\ell \in \N$ satisfying $\ell \rho \le_\Lambda \lambda$,
  and $f : \N \to \Lambda$ the map defined by $f(\ell) = \ell \rho$.
  Since $\Lambda$ is an affine monoid and $\rho$ is non-zero,
  the map $f$ is injective.
  Moreover, the image of $A$ by $f$ is contained in the finite subset $\Lambda^{\le \lambda}$.
  Hence $A$ is finite.
\end{proof}

\section{The main theorem} \label{sec : main}

Let $\Lambda$ be an affine monoid,
$\rho \in \Lambda$,
and $r \in \N$.
Assume that $\rho$ is reducible, that is, there exist $\sigma, \tau \in \Lambda_+$ such that $\rho = \sigma + \tau$,
and that $r \ge 2$.

Let us construct the additive monoid $\Lambda[\rho / r]$ which is added the $r$-th part of $\rho$ to $\Lambda$.
Consider the direct sum $\Lambda \oplus \N \alpha$ of $\Lambda$
and the free additive monoid $\N \alpha$ generated by a formal element $\alpha$,
and the equivalence relation $\sim$ on $\Lambda \oplus \N \alpha$ generated by
\[ (\lambda + \rho) + k \alpha \sim \lambda + (r + k) \alpha \qquad (\lambda \in \Lambda, \ k \in \N). \]
Define $\Lambda[\rho / r]$ to be the quotient $(\Lambda \oplus \N \alpha) / {\sim}$.
Since the equivalence relation $\sim$ is closed by addition of $\Lambda \oplus \N \alpha$, that is,
$x \sim x'$ and $y \sim y'$ implies $x + y \sim x' + y'$ for any $x, x', y, y' \in \Lambda \oplus \N \alpha$,
the quotient $\Lambda[\rho / r]$ has the canonical additive monoid structure.
We denote the equivalence class of $\lambda + k \alpha$ modulo $\sim$ simply by $\lambda + k \rho / r$.

Let $\pi : \Lambda \oplus \N \alpha \to \Lambda \times \N^{< r}$ be the map defined by
\[ \pi(\lambda + (\ell r + k) \alpha) = (\lambda + \ell \rho, k)
  \qquad (\lambda \in \Lambda, \ \ell \in \N, \ k \in \N^{< r}). \]
Then $x \sim y$ is equivalent to $\pi(x) = \pi(y)$ for any $x, y \in \Lambda \oplus \N \alpha$.
We can check
that there is a bijection $\Lambda \times \N^{< r} \cong \Lambda[\rho / r]$
which sends $(\lambda, k)$ to $\lambda + k \rho / r$, and
that $\Lambda[\rho / r]$ is an affine monoid.
Note that for $\lambda, \lambda' \in \Lambda$ and $k, k' \in \N^{< r}$ we have
\[ \lambda + k \rho / r \le_{\Lambda[\rho / r]} \lambda' + k' \rho / r \iff
\begin{cases}
  \lambda \le_\Lambda \lambda' & \text{if $k \le k'$,} \\
  \lambda + \rho \le_\Lambda \lambda' & \text{if $k > k'$.}
\end{cases} \]

Define the function $\ell_\rho : \Lambda \to \N$ by
\[ \ell_\rho(\lambda) = \max \{\, \ell \in \N \mid \ell \rho \le_\Lambda \lambda \,\}
  \qquad (\lambda \in \Lambda). \]
By Proposition~\ref{prp : Archi}, the function $\ell_\rho$ is well-defined since $\rho \in \Lambda_+$.
Note that $\ell_\rho$ satisfies that
\[ \ell \rho \le_\Lambda \lambda \iff \ell \le \ell_\rho(\lambda)
  \qquad (\ell \in \N, \ \lambda \in \Lambda). \]
Consequently,
\[ \ell_\rho(\lambda + \rho) = \ell_\rho(\lambda) + 1
  \qquad (\lambda \in \Lambda) \]
holds since $\ell \rho \le_\Lambda \lambda$ is equivalent to $(\ell + 1) \rho \le_\Lambda \lambda + \rho$.

The following is the main theorem of this paper.

\begin{thm} \label{thm : main}
  For $\lambda \in \Lambda$ and $k \in \N^{< r}$
  the Frobenius complex of $\Lambda[\rho / r]$ satisfies
  \[ \F(\lambda + k \rho / r; \Lambda[\rho / r]) \simeq
  \begin{cases}
    \bigvee\limits_{\ell = 0}^{\ell_\rho(\lambda)} \susp^{2 \ell + k} \F(\lambda - \ell \rho; \Lambda)
    & \text{if $k \le 1$,} \\
    \pt & \text{if $k \ge 2$.}
  \end{cases} \]
\end{thm}

We will prove Theorem~\ref{thm : main} in several steps.

\begin{prp}
  For $\lambda \in \Lambda$ and $k \in \N^{< r}$
  the Frobenius complex of $\Lambda[\rho / r]$ satisfies
  \[ \F(\lambda + k \rho / r; \Lambda[\rho / r]) \simeq
  \begin{cases}
    \F(\lambda + k \rho / 2; \Lambda[\rho / 2]) & \text{if $k \le 1$,} \\
    \pt & \text{if $k \ge 2$.}
  \end{cases} \]
\end{prp}

\begin{proof}
  Let $f : \Lambda[\rho / r] \to \Lambda[\rho /2]$ and $g : \Lambda[\rho / 2] \to \Lambda[\rho / r]$ be the maps
  defined by
  \begin{align*}
    f(\lambda + k \rho / r) &= \lambda + \min \{ k, 1 \} \rho / 2
    && (\lambda \in \Lambda, \ k \in \N^{< r}) \quad \text{and} \\
    g(\lambda + k \rho / 2) &= \lambda + k \rho / r
    && (\lambda \in \Lambda, \ k \in \N^{< 2}).
  \end{align*}
  Then $f$ and $g$ are poset maps, and satisfy $g f \le \id_{\Lambda[\rho / r]}$ and $f g = \id_{\Lambda[\rho / 2]}$.
  Moreover, $f$ and $g$ induce poset maps between $\Lambda[\rho / r]_+$ and $\Lambda[\rho / 2]_+$.
  Let $\lambda \in \Lambda$ and $k \in \N^{< r}$.
  Note that
  \[ g f(\lambda + k \rho / r) = \lambda + k \rho / r \iff k \le 1. \]
  We now apply Lemma~\ref{lem : poset} to $f : \Lambda[\rho / r]_+ \to \Lambda[\rho / 2]_+$
  and $g : \Lambda[\rho / 2]_+ \to \Lambda[\rho / r]_+$.
  If $k \le 1$, then we have
  \begin{align*}
    \F(\lambda + k \rho / r; \Lambda[\rho / r])
    &= \cpx{\Lambda[\rho / r]_+^{< \lambda + k \rho / r}}
    \\&\simeq \cpx{\Lambda[\rho / 2]_+^{< f(\lambda + k \rho / r)}}
    \\&= \F(\lambda + k \rho / 2; \Lambda[\rho / 2]).
  \end{align*}
  If $k \ge 2$, then we have
  \[ \F(\lambda + k \rho / r; \Lambda[\rho / r]) = \cpx{\Lambda[\rho / r]_+^{< \lambda + k \rho / r}}
  \simeq \pt. \qedhere \]
\end{proof}

By the previous proposition, we need only consider the case $r = 2$.
We define some subsets and maps to observe $\Lambda[\rho / 2]$.
Let
\begin{align*}
  U_1 &= \Lambda[\rho / 2] \setminus \{ 0, \rho / 2 \}, \\
  U_2 &= \Lambda[\rho / 2]_{\ge \rho / 2}, \quad \text{and} \\
  U_{12} &= U_1 \cap U_2.
\end{align*}
Then $U_1, U_2$ and $U_{12}$ are upper subsets of $\Lambda[\rho / 2]$ and
satisfy $U_1 \cup U_2 = \Lambda[\rho / 2]_+$ and $U_{12} = \Lambda[\rho / 2]_{> \rho / 2}$.
Let $h : \Lambda[\rho / 2] \to \Lambda$ and $i : \Lambda \to \Lambda[\rho / 2]$ be the maps defined by
\begin{align*}
  h(\lambda + k \rho / 2) &= \lambda &&(\lambda \in \Lambda, \ k \in \N^{< 2}) \quad \text{and} \\
  i(\lambda) &= \lambda &&(\lambda \in \Lambda).
\end{align*}
Then $h$ and $i$ are poset maps, and satisfy $i h \le \id_{\Lambda[\rho / 2]}$ and $h i = \id_\Lambda$.
Moreover, $h$ and $i$ induces poset maps between $U_1$ and $\Lambda_+$.

\begin{prp} \label{prp : susp}
  For $\lambda \in \Lambda$
  the Frobenius complex of $\Lambda[\rho / 2]$ satisfies
  \[ \F(\lambda + \rho / 2; \Lambda[\rho / 2]) \simeq \susp \F(\lambda; \Lambda[\rho / 2]). \]
\end{prp}

\begin{proof}
  If $\lambda = 0$, then we have
  \[ \F(\rho / 2; \Lambda[\rho / 2]) = \cpx{(0, \rho / 2)_{\Lambda[\rho / 2]}}
  = S^{-1} = \susp S^{-2} = \susp \F(0; \Lambda[\rho / 2]). \]
  Assume that $\lambda \in \Lambda_+$.
  Then we have
  \[ \F(\lambda + \rho / 2; \Lambda[\rho / 2])
  = \cpx{\Lambda[\rho / 2]_+^{< \lambda + \rho / 2}}
  = \cpx{U_1^{< \lambda + \rho / 2}} \cup \cpx{U_2^{< \lambda + \rho / 2}}. \]
  Applying Lemma~\ref{lem : poset} to $h : U_1 \to \Lambda_+$ and $i : \Lambda_+ \to U_1$, we obtain
  \[ \cpx{U_1^{< \lambda + \rho / 2}} \simeq \pt \]
  since $i h(\lambda + \rho / 2) = \lambda <_{\Lambda[\rho / 2]} \lambda + \rho / 2$.
  On the other hand, we have
  \[ \cpx{U_2^{< \lambda + \rho / 2}} = \cpx{[\rho / 2, \lambda + \rho / 2)_{\Lambda[\rho / 2]}} \simeq \pt. \]
  We also have
  \begin{align*}
    \cpx{U_1^{< \lambda + \rho / 2}} \cap \cpx{U_2^{< \lambda + \rho / 2}}
    &= \cpx{U_{12}^{< \lambda + \rho / 2}}
    \\&= \cpx{(\rho / 2, \lambda + \rho / 2)_{\Lambda[\rho / 2]}}
    \\&\approx \cpx{(0, \lambda)_{\Lambda[\rho / 2]}}
    \\&= \F(\lambda; \Lambda[\rho / 2]),
  \end{align*}
  where the homeomorphism $\cpx{(0, \lambda)_{\Lambda[\rho / 2]}} \approx \cpx{(\rho / 2, \lambda + \rho / 2)_{\Lambda[\rho / 2]}}$
  is given by the poset isomorphism
  \[ (0, \lambda)_{\Lambda[\rho / 2]} \cong (\rho / 2, \lambda + \rho / 2)_{\Lambda[\rho / 2]} \]
  which sends $\gamma$ to $\gamma + \rho / 2$.
  Applying Lemma~\ref{lem : susp}, we obtain
  \begin{align*}
    \F(\lambda + \rho / 2; \Lambda[\rho / 2])
    &= \cpx{U_1^{< \lambda + \rho / 2}} \cup \cpx{U_2^{< \lambda + \rho / 2}}
    \\&\simeq \susp (\cpx{U_1^{< \lambda + \rho / 2}} \cap \cpx{U_2^{< \lambda + \rho / 2}})
    \\&\approx \susp \F(\lambda; \Lambda[\rho / 2]). \qedhere
  \end{align*}
\end{proof}

The following proposition completes the proof of Theorem~\ref{thm : main}.

\begin{prp} \label{prp : wedge}
  For $\lambda \in \Lambda$
  the Frobenius complex of $\Lambda[\rho / 2]$ satisfies
  \begin{equation}
    \label{eqn : main}
    \F(\lambda; \Lambda[\rho / 2]) \simeq \bigvee_{\ell = 0}^{\ell_\rho(\lambda)} \susp^{2 \ell} \F(\lambda - \ell \rho; \Lambda).
  \end{equation}
\end{prp}

\begin{proof}
  The proof is by induction on $\ell_\rho(\lambda)$.
  Assume that $\ell_\rho(\lambda) = 0$, that is, $\rho \not\le_\Lambda \lambda$.
  If $\lambda = 0$, the both sides of \eqref{eqn : main} are $S^{-2}$.
  Let $\lambda \in \Lambda_+$.
  Then $i : \Lambda \to \Lambda[\rho / 2]$ induces
  a poset isomorphism $(0, \lambda)_\Lambda \cong (0, \lambda)_{\Lambda[\rho / 2]}$.
  Hence
  \[ \F(\lambda; \Lambda[\rho / 2]) = \cpx{(0, \lambda)_{\Lambda[\rho / 2]}}
  \approx \cpx{(0, \lambda)_\Lambda} = \F(\lambda; \Lambda). \]

  Assume that $\ell_\rho(\lambda) \ge 1$, that is, $\rho \le_\Lambda \lambda$.
  We have
  \[ \F(\lambda; \Lambda[\rho / 2]) = \cpx{\Lambda[\rho / 2]_+^{< \lambda}}
  =\cpx{U_1^{< \lambda}} \cup \cpx{U_2^{< \lambda}}. \]
  Applying Lemma~\ref{lem : poset} to $h : U_1 \to \Lambda_+$ and $i : \Lambda_+ \to U_1$, we obtain
  \[ \cpx{U_1^{< \lambda}} \simeq \cpx{\Lambda_+^{< h(\lambda)}} = \F(\lambda; \Lambda) \]
  since $i h(\lambda) = \lambda$.
  On the other hand, we have
  \[ \cpx{U_2^{< \lambda}} = \cpx{[\rho / 2, \lambda)_{\Lambda[\rho / 2]}} \simeq \pt. \]
  We also have
  \begin{align*}
    \cpx{U_1^{< \lambda}} \cap \cpx{U_2^{< \lambda}}
    &=\cpx{U_{12}^{< \lambda}}
    \\&= \cpx{(\rho / 2, \lambda)_{\Lambda[\rho / 2]}}
    \\&\approx \cpx{(0, \lambda - \rho / 2)_{\Lambda[\rho / 2]}}
    \\&= \F((\lambda - \rho) + \rho / 2; \Lambda[\rho / 2])
    \\&\simeq \susp \F(\lambda - \rho; \Lambda[\rho / 2]).
  \end{align*}
  Since $\ell_\rho(\lambda - \rho) = \ell_\rho(\lambda) - 1$,
  we can apply the inductive hypothesis to $\lambda - \rho$,
  and conclude that
  \begin{align*}
    \susp \F(\lambda - \rho; \Lambda[\rho / 2])
    &\simeq \susp \bigvee_{\ell = 0}^{\ell_\rho(\lambda - \rho)} \susp^{2 \ell} \F((\lambda - \rho) - \ell \rho; \Lambda)
    \\&\simeq \bigvee_{\ell = 1}^{\ell_\rho(\lambda)} \susp^{2 \ell - 1} \F(\lambda - \ell \rho; \Lambda).
  \end{align*}

  Since $\rho$ is reducible, we can take $\sigma \in \Lambda$ such that $0 <_\Lambda \sigma <_\Lambda \rho$.
  We now show that the inclusion
  $\cpx{U_1^{< \lambda}} \cap \cpx{U_2^{< \lambda}} = \cpx{U_{12}^{< \lambda}} \hookrightarrow \cpx{U_1^{< \lambda}}$
  is homotopic to the constant map to $\sigma$.
  Note that
  \begin{align*}
    U_{12}^{< \lambda}
    &= (\rho / 2, \lambda)_{\Lambda[\rho / 2]}
    \\&= \{\, \mu \mid \mu \in [\rho, \lambda)_\Lambda \,\}
      \sqcup \{\, \mu + \rho / 2 \mid \mu \in (0, \lambda - \rho]_\Lambda \,\}, \quad \text{and} \\
    U_1^{< \lambda}
    &= \{\, \mu \mid \mu \in (0, \lambda)_\Lambda \,\}
      \sqcup \{\, \mu + \rho / 2 \mid \mu \in (0, \lambda - \rho]_\Lambda \,\}.
  \end{align*}
  Let $g_1, g_2, g_3, g_4 : U_{12}^{< \lambda} \to U_1^{< \lambda}$ be the maps defined by
  \begin{align*}
    g_1(\mu) &= \mu \\
    g_2(\mu) &= \mu \\
    g_3(\mu) &= \mu \\
    g_4(\mu) &= \sigma && (\mu \in [\rho, \lambda)_\Lambda) \quad \text{and} \\
    g_1(\mu + \rho / 2) &= \mu + \rho / 2 \\
    g_2(\mu + \rho / 2) &= \mu \\
    g_3(\mu + \rho / 2) &= \mu + \sigma \\
    g_4(\mu + \rho / 2) &= \sigma && (\mu \in (0, \lambda - \rho]_\Lambda).
  \end{align*}
  Then $g_1, g_2, g_3$ and $g_4$ are poset maps and satisfy
  \[ g_1 \ge g_2 \le g_3 \ge g_4. \]
  Hence
  \[ \text{(inclusion)} = \cpx{g_1} \simeq \cpx{g_2} \simeq \cpx{g_3} \simeq \cpx{g_4} = \text{(constant)}. \]
  Applying Lemma~\ref{lem : wedge}, we obtain
  \begin{align*}
    \F(\lambda; \Lambda[\rho / 2])
    &= \cpx{U_1^{< \lambda}} \cup \cpx{U_2^{< \lambda}}
    \\&\simeq \cpx{U_1^{< \lambda}} \vee \susp (\cpx{U_1^{< \lambda}} \cap \cpx{U_2^{< \lambda}})
    \\&\simeq \F(\lambda; \Lambda) \vee
    \susp \bigvee_{\ell = 1}^{\ell_\rho(\lambda)} \susp^{2 \ell - 1} \F(\lambda - \ell \rho; \Lambda)
    \\&\simeq \bigvee_{\ell = 0}^{\ell_\rho(\lambda)} \susp^{2 \ell} \F(\lambda - \ell \rho; \Lambda). \qedhere
  \end{align*}
\end{proof}

\section{Applications} \label{sec : app}

Let $\Lambda$ be an affine monoid, $\rho \in \Lambda$ and $r \in \N$.
Assume that $\rho$ is reducible and that $r \ge 2$.

\begin{thm} \label{thm : Poincare}
  The multi-graded Poincar\'e series of $\Lambda[\rho / r]$ satisfies
  \[ P_{\Lambda[\rho / r]}(t, \z)  = \frac{P_\Lambda(t, \z) \cdot (1 + t \z^{\rho / r})}{1 - t^2 \z^\rho}. \]
  In particular, if $r =2$, then
  \[ P_{\Lambda[\rho / 2]}(t, \z) = \frac{P_\Lambda(t, \z)}{1 - t \z^{\rho / 2}}. \]
\end{thm}

\begin{proof}
  Combining the equation~\eqref{eqn : LS} and Theorem~\ref{thm : main}, we obtain
  \[ \beta_i(\lambda + k \rho / r; \Lambda[\rho / r]) =
  \begin{cases}
    \sum\limits_{\ell = 0}^{\ell_\rho(\lambda)} \beta_{i - 2 \ell - k}(\lambda - \ell \rho; \Lambda)
    & \text{if $k \le 1$,} \\
    0 & \text{if $k \ge 2$}
  \end{cases} \]
  for $\lambda \in \Lambda$ and $k \in \N^{< r}$,
  where we let $\beta_i(\lambda; \Lambda) = 0$ for $i < 0$.
  Thus we have
  \begin{align*}
    P_{\Lambda[\rho / r]}(t, \z)
    &= \sum_{i \in \N} \sum_{\lambda \in \Lambda} \sum_{k = 0}^{r - 1}
    \beta_i(\lambda + k \rho / r; \Lambda[\rho / r]) t^i \z^{\lambda + k \rho / r}
    \\&= \sum_{i \in \N} \sum_{\lambda \in \Lambda} \sum_{k = 0}^1 \sum_{\ell = 0}^{\ell_\rho(\lambda)}
    \beta_{i - 2 \ell - k}(\lambda - \ell \rho; \Lambda) t^i \z^{\lambda + k \rho / r}
    \\&= \sum_{\ell = 0}^\infty \sum_{i \in \N} \sum_{\substack{\lambda \in \Lambda \\ \lambda \ge \ell \rho}}
    \sum_{k = 0}^1 \beta_{i - 2 \ell - k}(\lambda - \ell \rho; \Lambda) t^i \z^{\lambda + k \rho / r}
    \\&= \sum_{\ell = 0}^\infty \sum_{j \in \N} \sum_{\mu \in \Lambda} \sum_{k = 0}^1
    \beta_j(\mu; \Lambda) t^{j + 2 \ell + k} \z^{\mu + \ell \rho + k \rho / r}
    \\&= \sum_{j \in \N} \sum_{\mu \in \Lambda} \beta_j(\mu; \Lambda) t^j \z^\mu
    \cdot \sum_{\ell = 0}^\infty (t^2 \z^\rho)^\ell \cdot \sum_{k = 0}^1 (t \z^{\rho / r})^k
    \\&= \frac{P_\Lambda(t, \z) \cdot (1 + t \z^{\rho / r})}{1 - t^2 \z^\rho}. \qedhere
  \end{align*}
\end{proof}

\begin{cor}
  If the multi-graded Poincar\'e series $P_\Lambda(t, \z)$ of $\Lambda$ is given by a rational function,
  then so is $P_{\Lambda[\rho / r]}(t, \z)$.
\end{cor}

We say that \emph{$X$ has the homotopy type of a wedge of spheres}
if $X$ is $S^{-2}$ or a topological space which is either
\begin{itemize}
  \item empty,
  \item contractible,
  \item homotopy equivalent to a sphere $S^n$ for some $n \ge 0$, or
  \item homotopy equivalent to a wedge $\bigvee\limits_{i = 1}^m S^{n_i}$ of spheres for some $n_1, \dots, n_m \ge 0$.
\end{itemize}

Note that for such an $X$ the reduced Betti numbers are independent of the choice of a field $K$,
and the homotopy type is determined only by the reduced Betti numbers.

\begin{cor} \label{cor : wedge}
  Assume that there is a generating system $\{ \alpha_1, \dots, \alpha_d \}$ of $\Lambda$
  which satisfies $\alpha_i \le_\Lambda \rho$ for each $i$.
  If $\F(\lambda; \Lambda)$ has the homotopy type of a wedge of spheres for each $\lambda \in \Lambda$,
  then so has $\F(\gamma; \Lambda[\rho / r])$ for each $\gamma \in \Lambda[\rho / r]$.
\end{cor}

\begin{proof}
  We first show that $\F(\lambda; \Lambda)$ is $0$-connected for each $\lambda \in \Lambda_{> \rho}$.
  We can assume that each $\alpha_i$ is non-zero.
  For any element $\mu = \sum_i m_i \alpha_i$ of $\Lambda_+$,
  since $(m_1, \dots, m_d)$ is non-zero, $\alpha_i \le \mu$ holds for some $i$.
  Hence
  \[ \F(\lambda; \Lambda) = \cpx{(0, \lambda)_\Lambda} = \bigcup_{i = 1}^d \cpx{[\alpha_i, \lambda)_\Lambda}. \]
  Each $\cpx{[\alpha_i, \lambda)_\Lambda}$ is contractible and contains $\rho$
  since $\alpha_i \le_\Lambda \rho <_\Lambda \lambda$.
  Thus $\F(\lambda; \Lambda)$ is $0$-connected.

  Consider $X = \F(\lambda + k \rho / r; \Lambda[\rho / r])$ for $\lambda \in \Lambda$ and $k \in \N^{< r}$.
  If $k \ge 2$, then $X$ is contractible.
  Let $k \le 1$.
  If $\ell_\rho(\lambda) = 0$, then we have
  \[ X \simeq \susp^k \F(\lambda; \Lambda). \]
  Since $\F(\lambda; \Lambda)$ has the homotopy type of a wedge of spheres, so has $X$.
  Let $\ell_\rho(\lambda) \ge 1$, that is, $\rho \le_\Lambda \lambda$.
  Then we have
  \[ X \simeq \bigvee_{\ell = 0}^{\ell_\rho(\lambda)} \susp^{2 \ell + k} \F(\lambda - \ell \rho; \Lambda). \]
  If either $\lambda \in \Lambda_{> \rho}$ or $k = 1$ holds,
  then each $\susp^{2 \ell + k} \F(\lambda - \ell \rho; \Lambda)$ is a $0$-connected topological space
  which has the homotopy type of a wedge of spheres.
  Thus so has $X$.
  If both $\lambda = \rho$ and $k = 0$ hold, then we have
  \[ X \simeq \F(\rho; \Lambda) \vee \susp^2 \F(0; \Lambda) = \F(\rho; \Lambda) \vee S^0. \]
  Since $\F(\rho; \Lambda)$ is a non-empty topological space which has the homotopy type of a wedge of spheres,
  so has $X$.
\end{proof}

\section{Examples} \label{sec : exa}

By $\sub{a_1, \dots, a_g}$ we denote the submonoid of $\N$ generated by $a_1, \dots, a_g \in \N$.
We first show some convenient propositions.

\begin{prp} \label{prp : isom}
  Let $\Lambda$ be a finitely generated submonoid $\sub{a_1, \dots, a_g}$ of $\N$, $\rho$ a reducible element of $\Lambda$,
  and $r \in \N_{\ge 2}$.
  Assume that $\Lambda \subset \sub r$ and let $\rho = r b$.
  If $r$ and $b$ are relatively prime,
  then there exists a monoid isomorphism $\Lambda[\rho / r] \cong \Lambda + \sub b = \sub{a_1, \dots, a_g, b}$
  which sends $\lambda \in \Lambda$ to $\lambda$ and $\rho / r$ to $b$.
\end{prp}

\begin{proof}
  Let $\widetilde F : \Lambda \oplus \N \alpha \to \N$ be the monoid homomorphism which sends
  $\lambda \in \Lambda$ to $\lambda$ and $\alpha$ to $b$.
  Then $\widetilde F$ induces a surjective monoid homomorphism
  \[ F : \Lambda[\rho / r] \to \Lambda + \sub b \]
  which sends $\lambda \in \Lambda$ to $\lambda$ and $\rho / r$ to $b$.
  It suffices to show that $F$ is injective, that is,
  $\widetilde F(x) = \widetilde F(y)$ implies $x \sim y$ for any $x, y \in \Lambda \oplus \N \alpha$.
  Let $\lambda, \lambda' \in \Lambda$ and $k, k' \in \N$,
  and assume that $\lambda + k b = \lambda' + k' b$.
  We can assume that $k \le k'$.
  Then we have $\lambda - \lambda' = (k' - k) b$.
  Since $\lambda - \lambda'$ is a multiple of $r$, there exists $\ell \in \N$ such that $k' - k = \ell r$.
  Hence $\lambda - \lambda' = \ell r b = \ell \rho$.
  Thus we have
  \[ \lambda + k \alpha = (\lambda' + \ell \rho) + k \alpha \sim \lambda' + (\ell r + k) \alpha = \lambda' + k' \alpha.
    \qedhere \]
\end{proof}

\begin{prp}
  Let $\Lambda$ be a finitely generated submonoid of $\N$ and $p$ a positive integer.
  Then the Frobenius complex $\F(p \lambda; p \Lambda)$ of $p \Lambda$ is homeomorphic to $\F(\lambda; \Lambda)$
  for each $\lambda \in \Lambda$,
  and the multi-graded Poincar\'e series of $p \Lambda$ satisfies
  \[ P_{p \Lambda}(t, \z) = P_\Lambda(t, \z^p). \]
\end{prp}

\begin{proof}
  By the monoid isomorphism $\Lambda \cong p \Lambda$ which sends $\lambda$ to $p \lambda$,
  $\F(\lambda; \Lambda)$ is homeomorphic to $\F(p \lambda; p \Lambda)$ for each $\lambda \in \Lambda$.
  By the equation~\eqref{eqn : LS}, we have
  \[ \beta_i(\lambda; \Lambda) = \betti_{i - 2}(\F(\lambda; \Lambda))
    = \betti_{i - 2}(\F(p \lambda; p \Lambda)) = \beta_i(p \lambda; p \Lambda) \]
  for $i \in \N$ and $\lambda \in \Lambda$.
  Hence
  \begin{align*}
    P_{p \Lambda}(t, \z)
    &= \sum_{i \in \N} \sum_{\lambda \in \Lambda} \beta_i(p \lambda, p \Lambda) t^i \z^{p \lambda}
    \\&= \sum_{i \in \N} \sum_{\lambda \in \Lambda} \beta_i(\lambda, \Lambda) t^i (\z^p)^\lambda
    \\&= P_\Lambda(t, \z^p). \qedhere
  \end{align*}
\end{proof}

Let us calculate  the simplest case.

\begin{exa}
  The Frobenius complex $\F(n; \N)$ of $\N$ has the homotopy type of a wedge of spheres for each $n \in \N$,
  and the multi-graded Poincar\'e series of $\N$ satisfies
  \[ P_\N(t, \z) = 1 + t \z. \]
\end{exa}

\begin{proof}
  We have
  \begin{align*}
    \F(0, \N) &= S^{-2}, \\
    \F(1, \N) &= \cpx{(0, 1)_\N} = S^{-1}, \quad \text{and} \\
    \F(n, \N) &= \cpx{(0, n)_\N} = \cpx{[1, n)_\N} \simeq \pt \qquad (n \ge 2).
  \end{align*}
  Thus $\F(n; \N)$ has the homotopy type of a wedge of spheres for each $n \in \N$,
  and the multi-graded Poincar\'e series satisfies
  \begin{align*}
    P_\N(t, \z)
    &= \sum_{i \in \N} \sum_{n \in \N} \beta_i(n; \N) t^i \z^n
    \\&= \sum_{i \in \N} \sum_{n \in \N} \betti_{i - 2}(\F(n; \N)) t^i \z^n
    \\&= 1 + t \z. \qedhere
  \end{align*}
\end{proof}

We now calculate some examples using the result of the previous sections.

\begin{exa} \label{exa : two}
  Let $a$ and $b$ be integers with $2 \le a < b$,
  and assume that $b \notin \sub a$.
  Then the Frobenius complex $\F(\lambda; \sub{a, b})$ of $\sub{a, b}$ has the homotopy type of a wedge of spheres
  for each $\lambda \in \sub{a, b}$,
  and the multi-graded Poincar\'e series of $\sub{a, b}$ satisfies
  \[ P_{\sub{a, b}}(t, \z) = \frac{(1 + t \z^a) \cdot (1 + t \z^b)}{1 - t^2 \z^m}, \]
  where $m$ is the least common multiple of $a$ and $b$.
\end{exa}

\begin{proof}
  We first show the case where $a$ and $b$ are relatively prime.
  Let $\rho = a b \in \sub a$.
  By Proposition~\ref{prp : isom}, there exists a monoid isomorphism $\sub a [\rho / a] \cong \sub{a, b}$
  which sends $a$ to $a$ and $\rho / a$ to $b$.
  Note that $\sub a = a \N \cong \N$.
  Then we can apply Theorem~\ref{thm : Poincare} and Corollary~\ref{cor : wedge} to $\sub a [\rho / a] \cong \sub{a, b}$.
  Thus $\F(\lambda; \sub{a, b})$ has the homotopy type of a wedge of spheres for each $\lambda \in \sub{a, b}$,
  and the multi-graded Poincar\'e series satisfies
  \begin{align*}
    P_{\sub{a, b}}(t, \z)
    &= \frac{P_{\sub a}(t, \z) \cdot (1 + t \z^b)}{1 - t^2 \z^\rho}
    \\&= \frac{P_\N(t, \z^a) \cdot (1 + t \z^b)}{1 - t^2 \z^{a b}}
    \\&= \frac{(1 + t \z^a) \cdot (1 + t \z^b)}{1 - t^2 \z^{a b}}.
  \end{align*}

  We now turn to general cases.
  Let $d$ be the greatest common divisor of $a$ and $b$,
  and let $a = a' d$ and $b = b' d$.
  Then $a'$ and $b'$ are relatively prime,
  and satisfy $2 \le a' < b'$ by the assumptions $a < b$ and $b \notin \sub a$.
  Note that $\sub{a, b} = d \sub{a', b'}$.
  Thus $\F(\lambda; \sub{a, b})$ has the homotopy type of a wedge of spheres
  for each $\lambda \in \sub{a, b}$,
  and the multi-graded Poincar\'e series satisfies
  \[ P_{\sub{a, b}}(t, \z)
    = P_{\sub{a', b'}}(t, \z^d)
    = \frac{\bigl( 1 + t (\z^d)^{a'} \bigr) \cdot \bigl( 1 + t (\z^d)^{b'} \bigr)}{1 - t^2 (\z^d)^{a' b'}}
    = \frac{(1 + t \z^a) \cdot (1 + t \z^b)}{1 - t^2 \z^m}. \qedhere \]
\end{proof}

\begin{exa}
  Let $p, q$ and $r$ be relatively prime integers with $2 \le p < q < r$.
  Let $\Lambda$ be the submonoid $\sub{p q, p r, q r}$ of $\N$ generated by
  three elements $p q, p r$ and $q r$.
  Then the Frobenius complex $\F(\lambda; \Lambda)$ of $\Lambda$
  has the homotopy type of a wedge of spheres for each $\lambda \in \Lambda$,
  and the multi-graded Poincar\'e series of $\Lambda$ satisfies
  \[ P_\Lambda(t, \z)
    = \frac{(1 + t \z^{p q}) \cdot (1 + t \z^{p r}) \cdot (1 + t \z^{q r})}{(1 - t^2 \z^{p q r})^2}. \]
\end{exa}

\begin{proof}
  Let $\rho = p \cdot q r = r \cdot p q = q \cdot p r \in \sub{p q, p r}$.
  By Proposition~\ref{prp : isom}, there is a monoid isomorphism $\sub{p q, p r}[\rho / p] \cong \sub{p q, p r, q r} = \Lambda$
  which sends $\lambda \in \sub{p q, p r}$ to $\lambda$ and $\rho / p$ to $q r$.
  Then we can apply Theorem~\ref{thm : Poincare} and Corollary~\ref{cor : wedge} to $\sub{p q, p r}[\rho / p] \cong \Lambda$.
  Thus $\F(\lambda; \Lambda)$ has the homotopy type of a wedge of spheres for each $\lambda \in \Lambda$,
  and the multi-graded Poincar\'e series satisfies
  \[ P_\Lambda(t, \z)
    = \frac{P_{\sub{p q, p r}}(t, \z) \cdot (1 + t \z^{q r})}{1 - t^2 \z^\rho}
    = \frac{(1 + t \z^{p q}) \cdot (1 + t \z^{p r}) \cdot (1 + t \z^{q r})}{(1 - t^2 \z^{p q r})^2} \qedhere \]
\end{proof}

\begin{exa}
  Let $a$ be a positive even number and $d$ a positive odd number,
  and assume that $a + 2 d \notin \sub a$.
  Let $\Lambda$ be the submonoid $\sub{a, a + d, a + 2 d}$ of $\N$ generated by the arithmetic sequence $a, a + d, a + 2 d$.
  Then the Frobenius complex $\F(\lambda; \Lambda)$ of $\Lambda$ has the homotopy type of a wedge of spheres
  for each $\lambda \in \Lambda$,
  and the multi-graded Poincar\'e series of $\Lambda$ satisfies
  \[ P_\Lambda(t, \z) = \frac{(1 + t \z^a) \cdot (1 + t \z^{a + 2 d})}{(1 - t^2 \z^m) \cdot (1 - t \z^{a + d})}, \]
  where $m$ is the least common multiple of $a$ and $a + 2 d$.
\end{exa}

\begin{proof}
  Let $\rho = 2 (a + d) = a + (a + 2 d) \in \sub{a, a + 2 d}$.
  By Proposition~\ref{prp : isom}, there is
  a monoid isomorphism $\sub{a, a + 2 d}[\rho / 2] \cong \sub{a, a + d, a + 2 d} = \Lambda$
  which sends $\lambda \in \sub{a, a + 2 d}$ to $\lambda$ and $\rho / 2$ to $a + d$.
  Then we can apply Theorem~\ref{thm : Poincare} and Corollary~\ref{cor : wedge} to $\sub{a, a + 2 d}[\rho / 2] \cong \Lambda$.
  Thus $\F(\lambda; \Lambda)$ has the homotopy type of a wedge of spheres for each $\lambda \in \Lambda$,
  and the multi-graded Poincar\'e series satisfies
  \[ P_\Lambda(t, \z)
    = \frac{P_{\sub{a, a + 2 d}}(t, \z)}{1 - t \z^{a + d}}
    = \frac{(1 + t \z^a) \cdot (1 + t \z^{a + 2 d})}{(1 - t^2 \z^m) \cdot (1 - t \z^{a + d})}. \qedhere \]
\end{proof}

The following is an answer to a question raised by Clark and Ehrenborg (\cite{CE}, Question 6.4).

\begin{exa}
  Let $p$ and $q$ be relatively prime integers with $2 \le p < q$.
  For $n \ge 1$, let $\Lambda_n$ be the submonoid $\sub{p^n, p^{n - 1} q, \dots, p q^{n - 1}, q^n}$ of $\N$ generated by
  the geometric sequence $p^n, p^{n - 1} q, \dots, p q^{n - 1}, q^n$.
  Then the Frobenius complex $\F(\lambda; \Lambda_n)$ of $\Lambda_n$ has the homotopy type of a wedge of spheres
  for each $\lambda \in \Lambda_n$,
  and the multi-graded Poincar\'e series of $\Lambda_n$ satisfies
  \[ P_{\Lambda_n}(t, \z)
    = \frac{\prod\limits_{i = 0}^n (1 + t \z^{p^{n - i} q^i})}{\prod\limits_{i = 1}^n (1 - t^2 \z^{p^{n - i + 1} q^i})}. \]
  In particular, if $p = 2$, then
  \[ P_{\Lambda_n}(t, \z)
    = \frac{1 + t \z^{2^n}}{\prod\limits_{i = 1}^n (1 - t \z^{2^{n - i} q^i})}. \]
\end{exa}

\begin{proof}
  The proof is by induction on $n$.
  If $n = 1$, this follows from Example~\ref{exa : two}.
  Let $n \ge 2$.
  Note that $p \Lambda_{n - 1} = \sub{p^n, p^{n - 1} q, \dots, p q^{n - 1}}$
  and $\Lambda_n = p \Lambda_{n - 1} + \sub{q^n}$.
  Let $\rho = p \cdot q^n = q \cdot p q^{n - 1} \in p \Lambda_{n - 1}$.
  By Proposition~\ref{prp : isom}, there is a monoid isomorphism $p \Lambda_{n - 1}[\rho / p] \cong \Lambda_n$
  which sends $\lambda \in p \Lambda_{n - 1}$ to $\lambda$ and $\rho / p$ to $q^n$.
  We now show that $\rho \in p \Lambda_{n - 1}$ satisfies the assumption of Corollary~\ref{cor : wedge}.
  By the equation
  \[ q^n - p^n = (q - p) \sum_{i = 0}^{n - 1} p^{n - 1 - i} q^i \]
  we have
  \[ \rho = p \cdot q^n = p \cdot p^n + \sum_{i = 0}^{n - 1} (q - p) \cdot p^{n - i} q^i. \]
  Hence $p^{n - i} q^i \le_{p \Lambda_{n - 1}} \rho$ for each $i \in \{ 0, \dots, n - 1 \}$.
  Then we can apply Theorem~\ref{thm : Poincare} and Corollary~\ref{cor : wedge} to $p \Lambda_{n - 1}[\rho / p] \cong \Lambda_n$.
  Thus $\F(\lambda; \Lambda_n)$ has the homotopy type of a wedge of spheres for each $\lambda \in \Lambda_n$,
  and the multi-graded Poincar\'e series satisfies
  \begin{align*}
    P_{\Lambda_n}(t, \z)
    &= \frac{P_{p \Lambda_{n - 1}}(t, \z) \cdot (1 + t \z^{q^n})}{1 - t^2 \z^{p q^n}}
    \\&= \frac{P_{\Lambda_{n - 1}}(t, \z^p) \cdot (1 + t \z^{q^n})}{1 - t^2 \z^{p q^n}}
    \\&= \frac{\prod\limits_{i = 0}^{n - 1} \bigl( 1 + t (\z^p)^{p^{n - 1 - i} q^i} \bigr)}
    {\prod\limits_{i = 1}^{n - 1} \bigl( 1 + t^2 (\z^p)^{p^{n - i} q^i} \bigr)}
    \cdot \frac{1 + t \z^{q^n}}{1 - t^2 \z^{p q^n}}
    \\&= \frac{\prod\limits_{i = 0}^n (1 + t \z^{p^{n - i} q^i})}
    {\prod\limits_{i = 1}^n (1 + t^2 \z^{p^{n - i + 1} q^i})} \qedhere.
  \end{align*}
\end{proof}

\subsection*{Acknowledgement}

The author would like to show his greatest appreciation to his supervisor Prof.~Toshitake Kohno
for helpful suggestions and insightful comments.

\end{document}